\documentclass[11pt]{amsart}
\usepackage{amsmath,amssymb,latexsym}
\usepackage[T1]{fontenc}
\usepackage{enumerate}

\def\F{\mathbb F}
\def\U{\mathcal U}

\theoremstyle{plain}
\newtheorem{theorem}{Theorem}[section]
\newtheorem*{theorema}{Theorem 1}
\newtheorem{prop}[theorem]{Proposition}
\newtheorem{fact}[theorem]{Fact}
\newtheorem{lemma}[theorem]{Lemma}
\newtheorem{cor}[theorem]{Corollary}

\theoremstyle{definition}

\newtheorem{remark}[theorem]{Remark}

\newtheorem*{quest}{Question}

\def\pf{\par\noindent{\em Proof. }}

\title[Finite rank and pseudofinite groups]{Finite rank and pseudofinite groups}
\author{Daniel Palac\'in}
\thanks{Research partially supported by the program MTM2014-59178-P and the European Research Council grant 338821.}  
\date{\today}
\address{Einstein Institute of Mathematics, The Hebrew University of Jerusalem, Givat Ram 9190401, Jerusalem, Israel}
\email{daniel.palacin@mail.huji.ac.il}

\begin{document}

\begin{abstract}
It is proven that  an infinite finitely generated group cannot be elementarily equivalent to an ultraproduct of finite groups of a given Pr\"ufer rank. Furthermore, it is shown that an infinite finitely generated group of finite Pr\"ufer rank is not pseudofinite.
\end{abstract}

\maketitle

\section{Introduction}

A {\em pseudofinite} group is an infinite model of the first-order theory of finite groups, or equivalently, an infinite group is pseudofinite if it is elementary equivalent to an ultraproduct of finite groups. In a similar way, one can define pseudofinite fields, which were first considered by Ax \cite{Ax}. These two notions are closely related as exhibited by Wilson in \cite{Wil}, who showed that a simple pseudofinite group is elementarily equivalent to a Chevalley group over a pseudofinite field. Furthermore, it was later observed  by Ryten \cite{Ryten} that it is even isomorphic to such a group, see \cite[Proposition 2.14(i)]{EJMR} for a more detailed argument.

In a pseudofinite group any definable map is injective if and only if it is surjective. Hence, it can be easily noticed, by considering the definable map $x\mapsto n\cdot x$ for a suitable positive integer $n$, that an infinite finitely generated abelian group is not pseudofinite. Consequently, since centralizers of non-identity elements in free groups are infinite cyclic groups, free groups, and in particular finitely generated ones, are
not pseudofinite.  In fact, Sabbagh raised the following question, which so far remains open:

\begin{quest}
Are there infinite finitely generated groups which are pseudo-finite?
\end{quest}

One can of course reformulate this question and ask which families of finite groups admit an ultraproduct elementarily equivalent to an infinite finitely generated group. Using deep results from finite group theory,  Ould Houcine and Point \cite{OulPoi} have made substantial progress in this direction and towards the understanding of the structure of a possible example. They showed that an infinite finitely generated pseudofinite group cannot be solvable-by-(finite exponent), generalizing a previous result of Kh\'elif, see \cite[Proposition 3.14]{OulPoi} for more details. In particular, an infinite finitely generated cannot be elementarily equivalent to a family $\{G_i\}_{i\in\mathbb N}$ of finite groups such that the solvable radical $R(G_i)$ of $G_i$ has a given derived length and that $G_i/R(G_i)$ has a fixed finite exponent.

The aim of this note is to obtain a similar result for another well-understood family of finite groups, finite groups of bounded Pr\"ufer rank. Recall that a group is said to have finite {\em Pr\"ufer rank} $r$ if every finitely generated subgroup can be generated by $r$ elements, and $r$ is the least such integer.  We prove the following:

\begin{theorema}\label{T:Main}
There is no infinite finitely generated group which is elementarily equivalent to an ultraproduct of finite groups of a given Pr\"ufer rank.
\end{theorema}

This is deduced from a slightly stronger result where the assumption on the Pr\"ufer rank is relaxed, see Proposition \ref{P:Main}. There, it is assumed that only specific quotients have finite rank. As a consequence from this stronger result, we immediately obtain in Corollary \ref{C:Main} that there is no infinite finitely generated pseudofinite group of finite Pr\"ufer rank.

\section{Preliminaries on finite and pseudofinite groups}

Let $G$ be an arbitrary group. Given a subset $X$ of $G$ and a positive integer $k$, we write 
$$
X^{*k} = \{x_1\dots x_k \, : \, x_1,\ldots,x_k \in X\}
$$
and denote by $\langle X \rangle$ the subgroup of $G$ generated by $X$, {\it i.e.}
$$
\langle X \rangle = \bigcup_{k\in \mathbb N} (X\cup X^{-1})^{*k}.
$$ 
Given a word $w(\bar x)$ in the variables $\bar x = (x_1,\ldots,x_m)$, let $w(G)$ denote the set of elements $w(\bar g)$ of $G$ for $m$-tuples $\bar g$ from $G$.
 As $w(G)$ is invariant under ${\rm Aut}(G)$, the verbal subgroup $\langle w(G) \rangle$ is characteristic but not necessarily definable in the pure language of groups.  The width of the word $w(\bar x)$ in $G$ is the smallest integer $k$, if it exists, such that $\langle w(G) \rangle  = w(G)^{*k}$. In particular, if $w(\bar x)$ has finite width in $G$, then the verbal group generated by $w(G)$ is clearly definable.
 
We focus our attention on two families of words: The family of words $w_n(x)=x^n$ of $n$th powers for each positive integer $n$ and the family of simple commutator words $w_{c,n}(x_1,\ldots,x_{2^n})$ defined recursively on $n$ as follows
$$
w_{c,1}(x_1,x_2)=[x_1,x_2]=x_1^{-1}x_2^{-1}x_1x_2
$$
and
$$
w_{c,n+1}(x_1,\ldots,x_{2^{n+1}}) = \big[  w_{c,n}(x_1,\ldots,x_{2^n}) ,  w_{c,n}(x_{2^n+1},\ldots,x_{2^{n+1}}) \big].
$$ 

In a series of papers, Nikolov and Segal obtained deep results concerning the finite width of certain families of words. Of special interest for our purposes are \cite[Theorem 1.7]{NikSeg1} (see also \cite[Theorem 1.2]{NikSeg3}) and \cite[Theorem 2]{NikSeg2}, where they show the existence of positive integers $\delta_c(d)$ and $\delta_n(d)$ such that the words $w_{c,1}$ and $w_n$ have width $\delta_c(d)$ and $\delta_n(d)$ in any $d$-generator finite group, respectively. In \cite{OulPoi}, Ould Houcine and Point noticed that these results transfer to pseudofinite groups by showing that these words have finite width in any group that is elementarily equivalent to  a non-principal ultraproduct of $d$-generator finite groups, see Proposition 3.7 there. As a consequence, the verbal subgroups generated by $w_{c,1}(G)$ and $w_n(G)$ are definable. Furthermore, they proved the following in \cite[Proposition 3.12]{OulPoi}:

\begin{fact}\label{F:DerBurn}
Let $G$ be a finitely generated pseudofinite group. Then the words $w_{c,n}$ and $w_n$ have finite width in $G$. Moreover, the subgroups $G^{(n)}=\langle w_{c,n}(G) \rangle$ and $G^n = \langle w_n(G) \rangle$ are definable and have finite index.
\end{fact}

It is worth noticing that $G^n$ has finite index in $G$ by Zelmanov's positive solution to the restricted Burnside problem. Next, using the fact above we note the following:
 
\begin{lemma}\label{L:Sim}
Let $G$ be a finitely generated pseudofinite group and assume that it contains an infinite definable definably simple non-abelian normal subgroup $N$. Then, the group $N$ is a finitely generated pseudofinite group.
\end{lemma}

Here, by a {\em definably simple} group we mean a group that has no proper non-trivial normal subgroups definable in the pure language of groups. 

\begin{proof}
Let $G$ be a finitely generated pseudofinite group $G$ and assume that it contains an infinite definably simple non-abelian normal subgroup $N$ which is definable by some formula $\phi(x,\bar a)$. By Frayne's Theorem \cite[Corollary 4.3.13]{ChanKei}, there exists some non-principal ultraproduct of an infinite family $\{G_i\}_{i\in I}$ of finite groups in which $G$ is elementarily contained. Let $\bar a_i$ be a representative of the equivalence class $\bar a$ in $G_i$ and let $N_i$ be the set $\phi(G_i,\bar a_i)$. Since by \L os's Theorem, almost every subset $N_i$ is a normal subgroup of $G_i$, we may assume that every subset $N_i$ is indeed a normal subgroup of $G_i$. Furthermore, as $N$ is elementary equivalent to an ultraproduct of the family $\{N_i\}_{i\in I}$, we may assume by \cite[Corollary 2.10]{Urg} that every subgroup $N_i$ is a finite simple group. 

The action of $G_i$ by conjugation on each $N_{i}$ gives rise to an embedding from $G_i/C_{G_i}(N_{i})$ into the group of automorphisms ${\rm Aut}(N_{i})$ of $N_{i}$, yielding a subgroup containing the group of inner automorphisms ${\rm Inn}(N_{i})$ of $N_{i}$. It is clear that ${\rm Inn}(N_{i})$ is isomorphic to $N_{i}$ since the latter is centerless. Furthermore, we known by \cite[Chapter 4.B page 304]{Gor} that the group ${\rm Aut}(N_{i})/{\rm Inn}(N_{i})$ is solvable of derived length three, and then so is each $G_i/ (N_{i} \, C_{G_i}(N_{i}))$. 

Now, since the normal subgroup $N_{i} \, C_{G_i}(N_{i})$ is definable over $\bar a_i$, it can be expressed as a first-order property of the tuple $\bar a_i$ that $G_i$ modulo $N_{i} \, C_{G_i}(N_{i})$ is solvable of derived length three. Namely, it suffices to express that any element of $G_i$ of the form $w_{c,3}(x_1,\ldots,x_8)$
belongs to $N_{i} \, C_{G_i}(N_{i})$. Hence, we obtain by \L os's Theorem that the pseudofinite group $G/(N\, C_G(N))$ is also solvable of derived length three and so it is finite by Fact \ref{F:DerBurn}. Whence, the group $N\, C_G(N)$ is finitely generated and so is $N\, C_G(N)/C_G(N)$, which is clearly isomorphic to $N$, yielding the result. 
\end{proof}

Another crucial point in our proofs, which also relies on the work of Nikolov and Segal \cite{NikSeg1, NikSeg3}, is a result due to Wilson \cite{Wil2}.

\begin{fact}\label{F:Solvable}
There exists a formula $\phi_s(x)$ in the pure language of groups that defines the solvable radical of any finite group.
\end{fact} 
It is clear that Wilson's formula induces a characteristic subgroup in any pseudofinite group. {\it A priori} this subgroup might not be solvable as witnessed by an ultraproduct of finite solvable groups without a common bound on their derived lengths. Nevertheless, as a definable factor of a pseudofinite group is also pseudofinite, see \cite[Lemma 2.16]{OulPoi}, the existence of such formula allows us to split a pseudofinite group into a semisimple part and a solvable-like one.  Here, by a semisimple group we mean a group without non-trivial abelian normal subgroups. In contrast with being simple, notice that semisimplicity is an elementary property. Namely, a group $G$ is semisimple if and only if it satisfies a first-order sentence asserting that there is no non-trivial element which commutes with all its conjugates, since for any element $a$ the abelian normal subgroup $Z(C_G(a^G))$ of $G$ is always definable.

\section{Proof of the Theorem}
 We first need a lemma on finite simple groups, which of course may be well-known. Our proof is inspired by the one of \cite[Lemma 4.2]{ManSeg}. 

\begin{lemma}\label{L:Fin2rk}
If $S$ is a finite simple non-abelian group of finite $2$-rank $r$, then either $|S|$ is $r$-bounded or $S$ is a finite simple group of Lie type of $r$-bounded Lie rank over a finite field.
\end{lemma}
Here and throughout the paper, by {\em $r$-bounded} we mean that there is a bound depending only on $r$.
\pf We know by the Classification of Finite Simple Groups that a finite simple non-abelian group is either an alternating group, a simple group of Lie type or one of the finitely many sporadic groups. 

It is clear that we can bound the size of all sporadic groups. Moreover, since the alternating group ${\rm Alt}(n)$ contains an elementary abelian $2$-subgroup of rank $2\cdot \lfloor n/4 \rfloor$, the size of an alternating group of $2$-rank $r$ is also $r$-bounded.

Suppose that $S$ is a finite simple group of Lie type of Lie rank $n$ over the finite field $\F_q$ and assume further that it has $2$-rank $r$. To bound the Lie rank in terms of $r$, it suffices to assume that $S$ is a classical simple group. In that case, we have that $S$ contains ${\rm PSL}_m(\F_q)$, where $m = \max\{2,\lfloor\frac{1}{2}(n-1)\rfloor\}$.
Thus, since ${\rm Alt}(m)$ embeds of into ${\rm PSL}_m(\mathbb F_q)$ via permutation matrices, the paragraph above implies that $m$ and hence $n$ are $r$-bounded. \qed

\begin{lemma}\label{L:Semi-simRank}
A semisimple pseudofinite group of finite $2$-rank has an infinite definable non-abelian simple  normal subgroup which is a linear group.
\end{lemma}
\begin{proof}
Assume that $G$ is a semisimple pseudofinite group of $2$-rank at most $r$. Thus, it is elementarily equivalent to an ultraproduct of an infinite family $\{F_i\}_{i\in I}$ of finite groups with respect to a non-principal ultrafilter $\U$ on the set $I$. After shrinking $I$ and $\U$ if necessary, we assume that all $F_i$ are semisimple and of $2$-rank at most $r$. Notice that the latter property is clearly expressible by a first-order sentence asserting the non existence of $r+1$ elements of order $2$ which commute pairwise. 
 
Now, set ${\rm Soc}(F_i)$ to denote the socle of $F_i$, the subgroup of $F_i$ generated by all minimal normal subgroups. Thus $$
{\rm Soc}(F_i) = S_{i,1} \times \ldots \times S_{i,k_i},
$$
where each $S_{i,j}$ is a non-abelian simple normal subgroup of $F_i$. Note that for every $i$ we have $k_i \le r$ by the Feit-Thompson Theorem, since each $S_{i,j}$ contains an element of order $2$ and any two $S_{i,j}$ and $S_{i,l}$ commute. Furthermore, by Lemma \ref{L:Fin2rk} there is a positive integer $\delta(r)$, depending only on $r$, such that either $|S_{i,j}| \le \delta(r)$ or $S_{i,j}$ is a simple group of Lie type of Lie rank at most $\delta(r)$. 
On the other hand, as we are assuming that $G$ is infinite, the size of the groups $F_i$ is unbounded with respect to $\U$. Hence, since $C_{F_i}({\rm Soc}(F_i))=1$, the group $F_i$ acts faithfully by conjugation on ${\rm Soc}(F_i)$ and so it embeds into the group of automorphisms of ${\rm Soc}(F_i)$. Thus, the size of the socle ${\rm Soc}(F_i)$ must be unbounded with respect to $\U$ as well. Consequently, as there is only a finite number of possible Lie ranks, there is some $\tau\le\delta(r)$ and a subset $I_0$ in $\U$ such that for each $i\in I_0$ some $S_{i,j_i}$ is a simple group of a fixed Lie type of Lie rank $\tau$ and moreover, the size of all these $S_{i,j_i}$ is unbounded with respect to $\U$. 

Next, we see that every $S_{i,j_i}$ with $i\in I_0$ can be defined in a uniform way. To do so, we use \cite[Theorem 4.2]{MacTent}, which asserts the existence of a positive integer $m=m(\tau)$ such that any finite simple group of a fixed Lie type of Lie rank $\tau$ is equal to the set $(C\cup C^{-1})^{*m}$ for any non-trivial conjugacy class $C$ with $|C|> m$. Hence, using this result, we find a formula $\psi(x,y)$ in the language of groups such that for some non-trivial element $a_i$ in $S_{i,j_i}$ we have that $\psi(F_i, a_i) = S_{i,j_i}$ and such that the $S_{i,j_i}$-conjugacy class $C_i$ of any non-trivial element of $F_i$ satisfying $\psi(x,a_i)$ generates $S_{i,j_i}$ in $m$ many steps, {\it i.e.} $S_{i,j_i} = (C_i\cup C_i^{-1})^{*m}$, provided that $|C_i|>m$.

By \L os's Theorem, we can find an element $a$ in $G$ such that the subset $\psi( G,a)$ is an infinite definable normal subgroup of $G$ and that $\psi(G,a) = (C\cup C^{-1})^{*m}$ for any $\psi(G,a)$-conjugacy class $C$ of every non-trivial element of $\psi(G,a)$ with $|C|>m$. In fact, since $\psi(G,a)$ is a definably simple group, any $\psi(G,a)$-conjugacy class $C$ of a non-trivial element of $\psi(G,a)$ must be infinite, as otherwise the centralizer of $C$ in $\psi(G,a)$ would be a definable (finite index) normal subgroup, yielding a contradiction. As a consequence, it then follows that the group $\psi(G,a)$ is simple. 

Finally, for the second part of the statement, note that $\psi(G,a)$ is elementarily equivalent to the ultraproduct $\prod_{\mathcal U} S_{i,j_i}$ and so it is elementarily embeddable into an ultrapower of $\prod_\U S_{i,j_i}$ by Frayne's Theorem \cite[Corollary 4.3.13]{ChanKei}. By a result of Point \cite{Poi}, such an ultrapower is isomorphic to a simple group of Lie type over a pseudofinite field. Hence, the 
simple group $\psi(G,a)$ is a linear subgroup.
\end{proof}

\begin{remark}
In fact, the proof of the lemma yields the existence of a simple group of Lie type by Ryten's results. Namely, we have shown that the simple group $\psi(G,a)$ is elementarily equivalent to the ultraproduct $\prod_{\mathcal U} S_{i,j_i}$ of simple groups $S_{i,j_i}$ of a fixed Lie type. Therefore, applying \cite[Proposition 2.14(i)]{EJMR} we get that $\psi(G,a)$ is a simple  group of Lie type.
\end{remark}

Before proving the main theorem, we deduce a weak version of Platonov's Theorem for pseudofinite groups.

\begin{cor}\label{C:Platonov}
Any pseudofinite group of finite Pr\"ufer rank is virtually elementarily equivalent to an ultraproduct of finite solvable groups.
\end{cor}
\begin{proof}
Let $G$ be a pseudofinite group of finite Pr\"ufer rank, and consider the definable characteristic subgroup $\phi_s(G)$ of $G$. Set $\bar G$ to denote $G/\phi_s(G)$, a semisimple group of finite Pr\"ufer rank, and suppose that it is infinite. By Lemma \ref{L:Semi-simRank}, we know that $\bar G$ contains an infinite simple normal subgroup $N$ which is a linear group. Hence, it is solvable-by-finite by Platonov's Theorem, see for instance \cite[Theorem 10.9]{We}, and consequently it is finite, a contradiction. Therefore, we obtain that $\bar G$ must be finite, which yields the result by Fact \ref{F:Solvable}. \end{proof}

In order to show Theorem \ref{T:Main}, we shall state and prove the following slightly stronger result. Before that, recall that the {\em upper rank} of a group is defined as the maximum, if it exists, of the Pr\"ufer rank of all its finite quotients.

\begin{prop}\label{P:Main}
There is no finitely generated pseudofinite group $G$ such that $G/\phi_s(G)$ has finite $2$-rank and that $\phi_s(G)$ has finite upper rank.
\end{prop}

\begin{proof}
 Suppose, towards a contradiction, that $G$ is a finitely generated pseudofinite group witnessing a contradiction, and consider the definable characteristic subgroup $\phi_s(G)$ of $G$. Set $\bar G$ to denote $G/\phi_s(G)$, a finitely generated semisimple group of $2$-rank $r$.

The proof is divided in two parts:

{\em First part (The semisimple case).} We aim to prove that $\bar G$ is finite. Assume otherwise that $\bar G$ is infinite. Thus, it contains an infinite definable simple normal subgroup $N$ which is as well a linear group by Lemma \ref{L:Semi-simRank}. Additionally, it is finitely generated by Lemma \ref{L:Sim}. Hence, by Malcev's Theorem it is residually finite, see for instance \cite[Corollary 4.4]{We}, and consequently it is finite, a contradiction. Therefore, we obtain that $\bar G$ must be finite, as desired.

{\em Second part (The solvable-like case).} After the first part, the subgroup $\phi_s(G)$ has finite index and so it is a finitely generated pseudofinite group of finite upper rank. Hence, we may assume that $G$ and $\phi_s(G)$ coincide, \textit{i.e.} the group $G$ is elementarily equivalent to a non-principal ultraproduct of an infinite family $\{H_i\}_{i\in I}$ of finite solvable groups. 
 
The subgroups $G^{(n)}$ forming the derived series of $G$, as well as the subgroups $G^n$ generated by the $n$th powers, are all definable and have finite index in $G$, by Fact \ref{F:DerBurn}. As each $H_i$ is solvable and the words $w_{c,1}$ have finite width in $G^{(n)}$ also by Fact \ref{F:DerBurn}, an easy application of \L os's Theorem yields that $G^{(n+1)}$ is properly contained in $G^{(n)}$. In particular, the group $G$ has arbitrary large finite quotients.

We claim that each finite quotient of $G$ is solvable. To do so, consider a finite index normal subgroup $N$ of $G$ and notice that it is definable, since it contains the finite index subgroup $G^k$, where $k=[G:N]$, which is also definable. Let $\chi(x,\bar y)$ be a formula in the pure language of groups such that $\chi(G,\bar c)=N$ for some finite tuple $\bar c$ from $G$. By \L os's Theorem, for almost all indices $i\in I$ (with respect to $\U$), we can find some  tuple $\bar c_i$ in $H_i$ such that $\chi(H_i,\bar c_i)$ is a normal subgroup of $H_i$ of index  $k$. Hence, since each $H_i$ is solvable, the $k$th term $H_i^{(k)}$ of the derived series  of $H_i$ must be contained in $\chi(H_i,\bar c_i)$ and so $G^{(k)}$ is contained in $N$ by \L os's Theorem.

To conclude, consider the characteristic subgroup $G^0$ of $G$ defined as the intersection of all finite index normal subgroups of $G$. Note that each finite quotient of $G/G^0$ is solvable and has Pr\"ufer rank at most $r$ by assumption. Thus, using the main result from \cite{Seg} we obtain that $G/G^0$ is nilpotent-by-abelian-by-finite. Therefore, the quotient $G/G^0$ is solvable and hence it must be finite, a contradiction. This finishes the proof. 
\end{proof}

As an immediate consequence we deduce the following:
\begin{cor}\label{C:Main}
There is no finitely generated pseudofinite group of finite Pr\"ufer rank.
\end{cor}
We finally give the proof of the main theorem.
\begin{proof}[Proof of Theorem \ref{T:Main}]
Let $G$ be an infinite finitely generated group and suppose towards a contradiction that there is an infinite family $\{G_i\}_{i\in I}$ of finite groups of Pr\"ufer rank $r$ such that $G$ is elementarily equivalent to $\prod_\U G_i$ for some non-principal ultrafilter $\U$ on $I$. 
Consider Wilson's formula $\phi_s(x)$. Note that $G_i/\phi_s(G_i)$ has $2$-rank at most $r$ and then, so does the pseudofinite group $G/\phi_s(G)$ by \L os's Theorem. Thus, the previous proposition yields that $\phi_s(G)$ has finite index in $G$. Hence, after replacing $G$ by $\phi_s(G)$ if necessary, we may assume that each $G_i$ is a solvable group of Pr\"ufer rank at most $r$. Moreover, as remarked in the solvable-like case of the proof of the previous proposition, we have that every finite index subgroup $N$ of $G$ is definable, and so $G$ has upper rank at most $r$. Therefore, applying once more Proposition \ref{P:Main} we obtain the desired contradiction.
\end{proof}

\bibliographystyle{plain}

\end{document}